\documentclass{iopart}
\usepackage{iopams}
\usepackage{geometry}                
\usepackage{graphicx}
\usepackage{epstopdf}
\usepackage{amsfonts}
\usepackage{amssymb}
\usepackage{amsthm}
\usepackage[pagebackref,pdfusetitle]{hyperref}
\usepackage{url}
\usepackage[percent]{overpic}

\def\eqalign#1{\null\,\vcenter{\openup\jot \mathsurround=0pt \ialign{\strut
     \hfil$\displaystyle{##}$&$ \displaystyle{{}##}$\hfil \crcr#1\crcr}}\,}
\def\eqref#1{(\ref{#1})}

\newtheorem{proposition}{Proposition}
\newtheorem{corollary}[proposition]{Corollary}
\newtheorem{definition}{Definition}
\newtheorem{example}{Example}

\def\R{\mathbb{R}}
\def\H{\mathsf{H}}

\begin{document}

\title{Discretization of polynomial vector fields by polarization }

\author{Elena Celledoni$^1$, Robert I McLachlan$^2$, David I McLaren$^3$, Brynjulf Owren$^1$ and G R W Quispel$^3$}

\address{$^1$ 	Department of Mathematical Sciences,
	NTNU,
	7491 Trondheim,
	Norway\eads{\mailto{elenac@math.ntnu.no}, \mailto{bryn@math.ntnu.no}}}
\address{$^2$ 	Institute of Fundamental Sciences,
	Massey University,
	Private Bag 11 222, Palmerston North 4442, New Zealand\ead{r.mclachlan@massey.ac.nz}}
\address{$^3$ 	Department of Mathematics,
	La Trobe University,
	Bundoora, VIC 3083, Australia\ead{elenac@math.ntnu.no}}

\begin{abstract}
\noindent
A novel integration method for quadratic vector fields was introduced by Kahan in 1993. Subsequently, it was shown that Kahan's method preserves a (modified) measure and energy when applied to quadratic Hamiltonian vector fields. Here we generalize Kahan's method to cubic resp. higher degree polynomial vector fields and show that the resulting discretization also preserves modified versions of the measure and energy when applied to cubic resp. higher degree polynomial Hamiltonian vector fields.
\end{abstract}

\section{Introduction: Kahan's method for quadratic vector fields}
The study of ordinary differential equations (ODEs) goes back centuries, to the time of Newton, Bernoulli, Euler, and contemporaries. Since the invention of the computer in the 1940s much attention has been devoted to the best ways to discretize differential equations so that they can be solved numerically. Initially, the main emphasis was on all-purpose methods (defined for all ODEs), such as Runge--Kutta methods and linear multistep methods, and their quantitative accuracy. During the last two or three decades, however, interest has expanded to considering special classes of ODEs and purpose-built algorithms that preserve the special features of each class. These novel methods are not just {\em quantitatively}, but also {\em qualitatively} accurate. This has resulted in methods that preserve symmetries, first integrals, symplectic structure, measure, foliations, Lyapunov functions, etc. These methods are called geometric integration methods \cite{ha-lu-wa,le-re}.

In 1993, Kahan introduced a numerical integration method for quadratic differential equations. For the quadratic ODE
\begin{equation}
\label{eq:quad}
\dot x = f(x) := Q(x,x)+Bx+c,\quad x\in\R^n,
\label{eq:quadvf}
\end{equation}
(where $x\in\R^n$, $Q$ is an $\R^n$-valued symmetric bilinear  form, $B\in\R^{n\times n}$, and $c\in\R^n$) it is defined by the map 
\begin{equation}
x\mapsto x'\colon \frac{x'-x}{h} = Q(x,x') + \frac{1}{2}B(x+x')+ c
\label{eq:kahanmap}
\end{equation}
where $h$ is the time step.
The method (\ref{eq:kahanmap}) was introduced  in \cite{kahan} for two examples, a scalar Riccati equation and a 2-dimensional Lotka--Volterra system and written down in the general form (\ref{eq:kahanmap}) in \cite{ka-li} (see also references therein).
Kahan wrote in the prologue to \cite{kahan},

\begin{quote}\em``I have used these unconventional methods for 24 years without quite understanding why they work so well as they do, when they work. That is why I pray that some reader of these notes will some day explain the methods' behavior to me better than I can, and perhaps improve them.''
\end{quote}

Initially, the mystery only deepened, for the Kahan method did not at first sight fit into any of the standard methods of discretizing ODEs, nor into any of the new methods that were developed as the field of geometric numerical integration grew. 
Yet in some sense Kahan's prayer has been fulfilled. The Kahan method has been found to have remarkable geometric properties. Studies have shown that the Kahan method preserves complete integrability in many cases \cite{CMOQgeometricKahan,integrable,hone09tds,hi-ki,petrera10,petrera12,petrera13,pe-pf-su,pfadler}. For a large class of Hamiltonian systems the method has a conserved quantity (related to energy) and an invariant measure. It is the restriction of a Runge--Kutta method to quadratic vector fields \cite{CMOQgeometricKahan}. However, so far only a part of the observed behavior of the method has been accounted for, and  the `explanations' to a degree only raise new questions, for they reveal aspects of Runge--Kutta methods and of discrete integrability that were previously unknown and unsuspected. Maps derived from the Kahan method are  birational, with birational inverses; thus they are elements of the Cremona group of  birational automorphisms. The algebra, geometry, and dynamics of this group have been studied extensively \cite{di-fa,rerikh}, although the phenomena illustrated by the Kahan method are apparently new.

Just one of the unusual features of Kahan's method is that the formulation (\ref{eq:kahanmap}) is defined only for quadratic differential equations. Although its Runge--Kutta formulation is defined for all ODEs, the special geometric properties appear to  hold only in the quadratic case. Yet there is no apparent structure to the set of quadratic differential equations that would distinguish them in this way, especially in relation to the birational maps. In this paper we propose a natural generalization of Kahan's method to polynomial vector fields of higher degree and show that it does inherit some of the geometric properties---invariant measures, first integrals, and integrability---of Kahan's method in some cases. (An alternative generalization of Kahan's method to higher degree vector fields is considered in \cite{ho-to}.)

We first observe that  a homogeneous quadratic vector field $f(x)$ can be expressed in terms
of a bilinear form $Q(x,x)$, as in \eqref{eq:quad}, using the technique of {\em polarization}:
\begin{equation}
Q(x,x') = \frac{1}{2}\left( f(x+x') - f(x) - f(x')\right).
\end{equation}
Then the Kahan method can be obtained by polarizing the quadratic terms of the ODE, evaluating them at $(x,x')$, and by replacing the linear and constant terms by the midpoint approximation. 

Polarization is a map from a homogeneous polynomial to a symmetric multilinear form in more variables. For example, the polarization of the cubic $f(x)$ is the trilinear form
\begin{equation*}
F(x_1,x_2,x_3) = \frac{1}{6}\frac{\partial}{\partial \lambda_1}\frac{\partial}{\partial \lambda_2}\frac{\partial}{\partial \lambda_3}f(\lambda_1 x_1 +\lambda_2 x_2+\lambda_3 x_3)|_{\lambda=0}
\end{equation*}
where $x$, $x_1$, $x_2$, and $x_3$ are all vectors in $\R^n$. This is equal to $\frac{1}{6}$ of the coefficient of $\lambda_1\lambda_2\lambda_3$ in $f(\lambda_1 x_1 +\lambda_2 x_2+\lambda_3 x_3)$.
It satisfies
\begin{equation*}
F(x,x,x)=f(x).
\end{equation*}
For example, consider the case $x\in\R^3$ and write $x=(y,z,w)^T$. Then the polarization of $3 y^2 z$ is $y_1 y_2 z_3 + y_2 y_3 z_1 + y_3 y_1 z_2$ and the polarization of
$6 y z w$ is $y_1 z_2 w_3+y_2 z_3 w_1+y_3 z_1 w_2+y_1 z_3 w_2+y_3 z_2 w_1+y_2 z_1 w_3$.
Polarization was used in \cite{da-ow} to obtain linearly implicit, integral-preserving methods for Hamiltonian PDEs.

Polarization of a homogeneous vector field of degree $k+1$ will lead to a multilinear form in $k+1$ variables. We will call these variables $x_0,\dots,x_k$, where $x_k\in\R^n$. The generalization of the Kahan method that we consider in this paper is to evaluate this multilinear form
at $k+1$ consecutive time steps, leading to a $k$-step numerical integrator. In this way, the bilinear character of the Kahan method carries over to higher degrees. The treatment of the linear term $\dot x$ is no longer unique; here we consider the simplest possible option of discretizing $\dot x$ by $(x_k-x_0)/(kh)$.

\begin{definition}
Let $V=\R^n$ and let $F$ be the multilinear map from $V^{k+1}$ to $\R^n$  associated with the homogeneous polynomial differential equation
\begin{equation*} \dot x = F(x,x,\dots,x)\ \ ( =: f(x))
\end{equation*}
 of degree $k+1$ on $V$. 
The {\em polar map} associated with $f$ is the birational map on $V^k$ given by $(x_0,\dots,x_{k-1}) \mapsto (x_1,\dots,x_k)$ where $x_k$ is the solution of the linear system
\begin{equation}
\label{eq:polar}
\frac{x_k - x_0}{k h} =  F(x_0,\dots, x_k).$$
\end{equation}

\end{definition}

Note that as both sides of \eqref{eq:polar} are linear in $x_0$ and in $x_k$, Eq. \eqref{eq:polar}, the expressions for both $x_k$ as a function of $x_0,\dots,x_{k-1}$, and for $x_0$ as a function of $x_1,\dots,x_k$, are rational functions. Thus, like the Kahan map, the polar map is birational. However, it is expected that the multilinearity of \eqref{eq:polar} is more special than mere birationality; when $k>1$ there are many birational integrators formed from $f$ that are not multilinear. The multistep leapfrog method
\begin{equation*}
\frac{x_2 - x_0}{2h} = f(x_1)
\end{equation*}
is an example; maps of this form are not expected to have special geometric properties.

\begin{proposition}
The polar map of a homogeneous quadratic is its Kahan map. 
If a nonhomogeneous quadratic is suspended to a  homogeneous form in one dimension higher (e.g. if $\dot x = x^2 + b x + c$ is replaced by $\dot x = x^2 + b x y + c y^2$, $\dot y = 0$), then the polarization of the suspended vector field, projected to the original phase space,  is exactly the Kahan map of the nonhomogeneous quadratic.
\end{proposition}

\begin{proposition}
The polar map is (i) self-adjoint (in the sense of symmetric multistep methods \cite{ha-lu-wa}), and (ii) a general linear method restricted to vector fields that are homogeneous polynomials of degree $k+1$.
\end{proposition}
\begin{proof}
(i) Eq. (\ref{eq:polar}) is invariant under $(x_0,\dots,x_k,h) \mapsto (x_k,\dots,x_0,-h)$.\\
(ii) From a standard identity in algebraic polarization \cite[p. 110]{greenberg},
\begin{equation*}
F(x_0,\dots,x_{k}) = \frac{1}{(k+1)!} \sum_{1\le m \le k+1 \atop 0\le i_1<\dots < i_m\le k} (-1)^{k+1-m} f(x_{i_1}+\dots+x_{i_m}),
\end{equation*}
where the sum is over all nonempty subsets of $\{0,\dots,k\}$.
Using homogeneity of $f$, we get
\begin{equation*}
F(x_0,\dots,x_{k}) = \frac{1}{(k+1)!} \sum_{1\le m \le k+1 \atop 0\le i_1<\dots < i_m\le k} (-1)^{k+1-m} m^{k+1} f\left(
\frac{x_{i_1}+\dots+x_{i_m}}{m}\right).
\end{equation*}
There are $2^{k+1}-1$ nonempty subsets of $\{0,\dots,k\}$, so
in this form, the vector field $f(x)$ is evaluated at $2^{k+1}-1$ points, each of which is a convex combination of the $x_j$. These points may be taken to be the stage values of a general linear method.
\end{proof}

General linear methods are a natural class of methods that include both Runge--Kutta and linear multistep methods \cite{butcher}. In this case the method has $k-1$ `auxiliary variables' $x_0,\dots, x_{k-2}$ that are carried forward along with the `current point' $x_{k-1}$, but is `mono implicit' in the sense that only a single variable, $x_k$, enters nonlinearly. (When $f$ is degree $k+1$, $x_k$ even enters linearly.)

For example, if $f(x)$ is a homogeneous cubic then we can write
\begin{equation*}\fl
\eqalign{
F(x_0,x_1,x_2) &= \frac{1}{6}(f(x_0+x_1+x_2)-f(x_0+x_1)-f(x_0+x_2)-f(x_1+x_2) \cr
& \qquad +f(x_0)+f(x_1)+f(x_2)) \cr
&= \frac{27}{6}f\left(\frac{x_0+x_1+x_2}{3}\right) 
- \frac{8}{6}f\left(\frac{x_0+x_1}{2}\right) 
- \frac{8}{6}f\left(\frac{x_0+x_2}{2}\right) 
- \frac{8}{6}f\left(\frac{x_1+x_2}{2}\right) \cr
& \qquad+ \frac{1}{6}f(x_0)
+ \frac{1}{6}f(x_1)
+ \frac{1}{6}f(x_2).
}
\end{equation*}

The special behavior of the Kahan method is seen most easily on the scalar ODE $\dot x = x^2$, for which it yields the map $x_0\mapsto x_1 = x_0/(1 - h x_0)$, a M\"obius transformation which is easily integrated. It can be seen to converge past the singularity at $t=1/x(0)$. In contrast, an explicit method (like forward Euler) has no singularity, and an implicit method (like backward Euler) does not define a smooth map $\varphi\colon X\to X$ for any sensible domain $X\subset\R$. We first study the polar map associated to a higher-degree analog of this ODE.

\begin{proposition}
Let $k$ be a positive integer. 
The polar map of $\dot x = x^{k+1}$, $x\in\R$, is explicitly integrable.
\end{proposition}
\begin{proof}
Eq. (\ref{eq:polar}) written at time step $n$ becomes in this case
\begin{equation*}
x_{n+k} = x_n + h k\,  x_n x_{n+1} \dots x_{n+k}.
\end{equation*}
Dividing both sides by $x_n\dots x_{n+k}$,
\begin{equation*}
\frac{1}{x_n \dots x_{n+k-1}} = \frac{1}{x_{n+1}\dots x_{n+k}} + h k.
\end{equation*}
Thus, $I_n := 1/(x_n \dots x_{n+k-1})$ obeys
\begin{equation*}
I_n = I_{n-1} - h k
\end{equation*}
with solution
\begin{equation*}
I_n = I_0 - n h k.
\end{equation*}
Taking logs, $\log(x_n)$ obeys the linear, constant-coefficient, nonautonomous difference equation
\begin{equation*}
\log(x_n) + \dots + \log(x_{n+k-1}) = -\log(I_n)
\end{equation*}
 which is easily solved.
\end{proof}

This encouraging behaviour motivates our study of the polar map. Although not exhaustive, one large class of vector fields for which the Kahan map is known to have special properties is that of the Hamiltonian vector fields. For $k=1$ the polar (Kahan) map derived from Hamiltonian vector fields on Poisson spaces with constant Poisson structure is known to have a conserved quantity and an invariant measure \cite{CMOQgeometricKahan}. This explains their integrability in some (low-dimensional) cases.
We will now show how these phenomena generalize to the case $k>1$.

\section{Integrals of the polar map}
We will now consider the case when $f(x)$ is a homogeneous Hamiltonian vector field defined on a symplectic vector space. The following result establishes the existence of several first integrals of an iterate of the associated polar map. These integrals correspond to `modified energies' of the map as they all approximate the Hamiltonian in the limit of small step size.

\begin{proposition}
Let $H\colon V\rightarrow \mathbb{R}$, the Hamiltonian, be a  homogeneous polynomial of degree $k+2$ on $\mathbb{R}^n$ and let $\mathsf{H}$ be a symmetric $(k+2)$-tensor such that
$H(x)=\frac{1}{(k+2)!}\mathsf{H}(x,x,\dots,x)$. Let $\Omega$ be a constant invertible antisymmetric $n\times n$ matrix, and let $\omega$ be its associated symplectic form on $V\!$,
i.e.,
\begin{equation*}
\omega\colon V\times V\to\mathbb{R},\quad \omega(u,v) = u^T \Omega v.
\end{equation*}
Let $K = \Omega^{-1}$. 
Then
\begin{itemize}
\item[(i)]
the Hamiltonian ODE on $V$ associated with $(H,\Omega)$ is
\begin{equation}
\label{eq:KH}
\dot x =  \frac{1}{(k+1)!} K \mathsf{H}(x,\dots,x,\cdot);
\end{equation}
\item[(ii)]
the associated polar map is defined via
\begin{equation}
\label{eq:polarmapdegn}
\frac{x_k-x_0}{kh} = \frac{1}{(k+1)!}K \mathsf{H}(x_0,\dots,x_k,\cdot\,);
\end{equation}
\item[(iii)] the associated polar map \eqref{eq:polarmapdegn} has $k$ independent  $k$-integrals
\begin{equation}
\omega(x_0, x_1),\ \omega(x_1,x_2),\dots,\ \omega(x_{k-1},x_k).
\end{equation}
\end{itemize}
\end{proposition}
\begin{proof}
First note that the components of the gradient $\nabla H(x)$ are homogeneous polynomials of degree $k+1$ in the components of $x$ and we have
$$\frac{1}{(k+2)!}\mathsf{H}(x,x,\dots,x)=H(x)=\frac{1}{k+2}\,\nabla H(x)^T,$$
so
$$\frac{1}{(k+1)!}\mathsf{H}(x,\dots,x,\cdot)=\nabla H(x).$$
This yields (i) and (ii).

Recall that a $k$-integral of a map is an invariant of the $k$th iterate of the map \cite{ha-by-qu-ca}.
We first show that $\omega(x_0,x_1) = \omega(x_k,x_{k+1})$.
Writing $a = k h / (k+1)!$, we have
\begin{eqnarray*}
\fl
\omega(x_0,x_1)-\omega(x_k,x_{k+1}) &= \omega(x_k-a K \H(x_0,\dots,x_k,\cdot),x_{k+1} -a K \H(x_1,\dots,x_{k+1},\cdot))-\omega(x_k,x_{k+1}) \\
&=- a \omega(K \H(x_0,\dots,x_k,\cdot),x_{k+1})
+ a \omega(K \H(x_1,\dots,x_{k+1},\cdot),x_k) + \\
& \qquad \quad a^2 \omega(K \H(x_0,\dots,x_k,\cdot), K \H(x_1,\dots,x_{k+1},\cdot)) \\
&=   a  \H(x_0,\dots,x_k,x_{k+1}) - a \H(x_1,\dots,x_{k+1},x_k) \\
& \qquad + a^2 \H(x_1,\dots,x_{k+1},K \H(x_0,\dots,x_k,\cdot)) \\
&=  a  \H(x_0,\dots,x_k,x_{k+1}) - a \H(x_1,\dots,x_{k+1},x_k) \\
& \qquad + a^2  \H(x_1\dots,x_{k+1},(x_k-x_0)/a)\\
&= 0.
\end{eqnarray*}
Now if $I(x)$ is any $k$-integral of a map $\varphi$, then so is $I\circ\varphi^{(m)}$ for any integer $m$ \cite{ha-by-qu-ca}. This yields the remaining $k$-integrals and concludes the proof.
\end{proof}

The integrals all approach the Hamiltonian of the original system as $x_m \to x(m h)$, for in this limit  we have
\begin{equation*}
\eqalign{
\frac{1}{h(k+2)}\omega(x_m,x_{m+1}) &= \frac{1}{h(k+2)}\omega(x_m, x_{m+1}-x_m)\cr
& \to \frac{1}{k+2}\omega(x(m h),\dot x(m h)) \cr
&=  \frac{1}{k+2}\nabla H(x(mh))^T x(m h) \cr
&= H(x(m h)).
}
\end{equation*}
Note that the $k$-integral $\omega(x_{k-1},x_k)$ is a rational function of $(x_0,\dots,x_{k-1})$, for $x_k$ is defined through \eqref{eq:polarmapdegn}; the other $k$-integrals are all quadratic.

Analogous results hold for Poisson systems of the form \eqref{eq:KH} where $K$ is antisymmetric but not invertible. These can be established by performing a linear change of variables that puts $K$ in its Darboux normal form. Because the discretization method is linear, it commutes with this change of variables. 

\section{Invariant measure of the polar map}
The next proposition shows that the polar map in this case is also measure-preserving in the extended phase space. The measure is `modified' in the sense that it converges to the invariant measure of the ODE in the limit of small step size.

\begin{proposition}
Let $K$ be a constant antisymmetric $n\times n$ matrix and let $\H\colon V^{k+2}\to\R$ be multilinear. Let $\mu$ be a constant measure on $V$ and let $\mu^k$ be the corresponding product measure on $V^k$. Then the  map on $V^k$ induced by the polar map (\ref{eq:polar}) associated with the homogeneous Hamiltonian vector field $\dot x = K \H(x,x,\dots,x,\cdot)$ has the invariant measure
\begin{equation}
\label{eq:measure}
\frac{\mu^k}{\det (I-c K\, \mathsf{H}(x_0,\dots, x_{k-1},\cdot,\cdot))}
\end{equation}
where 
\begin{equation*}
c = \frac{h}{(k-1)!}.
\end{equation*}
\end{proposition}
\begin{proof}
First note that we have
$$\frac{1}{(k+2)!}\mathsf{H}(x,x,\dots,x)=H(x)=\frac{1}{k+2}\,\nabla H(x)^Tx=\frac{1}{(k+2)(k+1)}\,x^TH''(x)x,$$
and 
$$\nabla H(x)=\frac{1}{k+1}\,H''(x)\,x.$$
So
$$\frac{1}{(k+1)!}\mathsf{H}(x,\dots,x,\cdot)=\nabla H(x),\qquad \frac{1}{k!}\mathsf{H}(x,\dots, x,\cdot,\cdot)=H''(x).$$

Let $X=[x_0^T,\dots , x_{k-1}^T]^T$. We want to prove that the Jacobian of the map
$$\varphi\colon (x_0,\dots,x_{k-1})\mapsto (x_0',\dots , x_{k-1}') \,\left\{  \begin{array}{lcl}
x_0'&=&x_1\\
x_1'&=&x_2\\
 & \vdots & \\
 x_{k-1}' &=& x_k=x_0+\frac{kh}{(k+1)!}K\mathsf{H}(x_0,\dots,x_k,\cdot\,)
 \end{array}
\right .$$
has the following determinant
$$\det \frac{\partial \varphi}{\partial {X}}=\frac{\det (I-c K\, \mathsf{H}(x_1,\dots, x_k,\cdot,\cdot))}{ \det (I - c K\,\mathsf{H}(x_0,\dots, x_{k-1},\cdot,\cdot))}.$$
We first observe that
$$\det \frac{\partial \varphi}{\partial {X}}=\det \frac{\partial x_{k-1}'}{\partial x_0}=\det \frac{\partial x_{k}}{\partial x_0}.$$
This follows directly from the format of the Jacobian and in fact
$$\det \frac{\partial \varphi}{\partial {X}}=\det \left[
\begin{array}{ccccc}
O & I & O &\dots & O\\
O & O& I & \ddots& O\\
\vdots &\ddots& \ddots & \ddots &\vdots\\
O      & O & \dots & O & I\\
\frac{\partial x_k}{\partial x_0} &\frac{\partial x_k}{\partial x_1}& \dots & \frac{\partial x_k}{\partial x_{k-2}} & \frac{\partial x_k}{\partial x_{k-1}}\\
\end{array}
\right]=\det \frac{\partial x_k}{\partial x_0}.$$

Differentiating \eqref{eq:polarmapdegn} on both sides with respect to $x_0$, and using the symmetry of $\mathsf{H}$ we have
$$\frac{\partial x_k}{\partial x_0}= I +c K\, \mathsf{H}(x_0,\dots, x_{k-1},\cdot,\cdot)\,\frac{\partial x_k}{\partial x_0}+c K\, \mathsf{H}(x_1,\dots, x_{k},\cdot,\cdot)\,\frac{\partial x_0}{\partial x_0}.$$
Rearranging the terms we obtain
$$\frac{\partial x_k}{\partial x_0}=(I- c K\, \mathsf{H}(x_0,\dots, x_{k-1},\cdot,\cdot))^{-1} (I+c K\,\mathsf{H}(x_1,\dots, x_k,\cdot,\cdot)),$$
and
$$\det \frac{\partial x_k}{\partial x_0}=\frac{\det (I+c K\, \mathsf{H}(x_1,\dots, x_k,\cdot,\cdot))}{\det (I-c K\, \mathsf{H}(x_0,\dots, x_{k-1},\cdot,\cdot))}.$$
Using $\det (A)= \det (A^T)$ and the Sylvester determinant theorem $\det(I+AB)=\det(I+BA)$ in the numerator, we obtain
$$\det \frac{\partial x_k}{\partial x_0}=\frac{\det (I-c K\, \mathsf{H}(x_1,\dots, x_k,\cdot,\cdot))}{\det (I-c K\, \mathsf{H}(x_0,\dots, x_{k-1},\cdot,\cdot))},$$
establishing the result.
\end{proof}

Note that in the case $k=1$, in which case \eqref{eq:polarmapdegn} reduces to the Kahan method for homogeneous cubic Hamiltonians, the invariant measure \eqref{eq:measure} can be written as
$\mu/\det(I-\frac{1}{2}h f'(x))$, which is the form of the invariant measure for the Kahan method
found in \cite{CMOQgeometricKahan}.

\section{Integrability of the polar map}
The next property of the polar map concerns a $(k-1)$-dimensional symmetry group, so it is a phenomenon that only appears for $k>1$.

\begin{proposition}
\begin{itemize}
\item[(i)]
The $k$th iterate of the polar map (\ref{eq:polarmapdegn}) is equivariant with respect to the scaling symmetry group $x_m \mapsto \lambda_m x_m$, $m=0,\dots, k-1$, where $\prod_{m=0}^{k-1}\lambda_m = 1$, i.e., the map \eqref{eq:polarmapdegn} has a $(k-1)$-dimensional $k$-symmetry group.
\item[(ii)] The measure (\ref{eq:measure}) is invariant under this scaling group. 
\item[(iii)] The integral $\prod_{m=0}^{k-1}\omega(x_m,x_{m+1})$ is invariant under this scaling group. 
\item[(iv)] When $k$ is even, the $2$-integrals 
\begin{equation*}
\omega(x_0,x_1)\omega(x_2,x_3)\dots \omega(x_{k-2},x_{k-1})
\end{equation*}
and
\begin{equation*}
\omega(x_1,x_2)\omega(x_3,x_4)\dots \omega(x_{k-1},x_k)
\end{equation*}
are invariant under this scaling group.
\end{itemize}
\end{proposition}
\begin{proof}
\begin{itemize}
\item[(i)]
Under the map $x_m \mapsto \lambda_m x_m$, the final equation defining the polar map, (\ref{eq:polarmapdegn}), is transformed to
\begin{equation*}
\lambda_0 x_k = \lambda_0 x_0 + k h K \frac{1}{(k+1)!} \H(\lambda_0 x_0,\lambda_1 x_1,\dots,
\lambda_{k-1}x_{k-1},\lambda_0 x_k,\cdot)
\end{equation*}
which is identical to \eqref{eq:polarmapdegn} under the condition $\prod_{m=0}^{k-1}\lambda_m = 1$. Therefore, the function $\varphi(x_0,\dots,x_{k-1})$ ($=x_k$) defined through the solution of \eqref{eq:polarmapdegn} scales as $\varphi(x_0,\dots,x_{k-1}) \mapsto \lambda_0 \varphi(x_0,\dots,x_{k-1})$.
The $k$th iterate of the polar map can be written
\begin{equation*}
\eqalign{
x_0^{(k)} &= \varphi(x_0,\dots,x_{k-1}) \cr
x_1^{(k)} &= \varphi(x_1,\dots,x_k) \cr
\dots  \cr
x_{k-1}^{(k)} &= \varphi(x_{k-1},\dots,x_{2k-2}) \cr
}
\end{equation*}
and each equation is invariant under the action $x_m \mapsto \lambda_{\mathrm{mod}(m,k)} x_m$ induced on the iterates of the map.
\item[(ii)] Follows from $\H(x_0,\dots,x_{k-1}) = \H(\lambda_0 x_0,\dots,\lambda_{k-1} x_{k-1})$.
\item[(iii)] We have
\begin{equation*}
\eqalign{
 \prod_{m=0}^{k-1} \omega(x_m,x_{m+1})
& \mapsto \prod_{m=0}^{k-1} \omega(\lambda_m x_m,\lambda_{m+1} x_{m+1}) \cr
& = \prod_{m=0}^{k-1}\lambda_m^2 \prod_{m=0}^{k-1} \omega(x_m,x_{m+1}) \cr
& = \prod_{m=0}^{k-1} \omega(x_m,x_{m+1}) \cr}
\end{equation*}
which establishes the result.
\item[(iv)] Under the symmetry, each of the given 2-integrals is multiplied by a factor 
$\prod_{m=0}^{k-1}\lambda_m$, which establishes the result.
\end{itemize}
\end{proof}

These results yield a 5-parameter family of integrable 4-dimensional rational maps.

\begin{corollary}
The polar map is completely integrable in the case $k=2$, $n=2$. 
\end{corollary}
\begin{proof}
The 2nd iterate of the polar map in this case has a 1-dimensional measure-preserving symmetry group. The map thus descends to a measure-preserving map on the 3-dimensional quotient \cite{huang}. The two integrals of the 2nd iterate of the polar map are invariant under the symmetry and hence also pass to the quotient. This yields a 3-dimensional measure-preserving map with 2 integrals, thus integrable \cite{bruschi}. The reconstruction dynamics obey a 1-dimensional, linear, constant-coefficient, nonautonomous difference equation and hence are integrable. From the integration of the 2nd iterate, the integration of the polar map itself is immediate.
\end{proof}

\begin{example}\rm
To be fully explicit we give here the integrable rational map obtained in the case $k=2$, $n=2$. Let
$(q,p)$ be coordinates on $V=\R^2$ and let the Poisson tensor and Hamiltonian be
\begin{equation*}
K=\left({\matrix{0 & 1 \cr -1 & 0}}\right),\quad H = a q^4 + 4 b q^3 p + 6 c q^2 p^2 + 4 d q p^3 + e p^4.
\end{equation*}
Then the polar map on $V^2$ is $(q_0,p_0,q_1,p_1)\mapsto (q_1,p_1,q_2,p_2)$, with
\begin{equation}
\label{eq:explicit}
\eqalign{q_2 &= \frac{q_0 + 2 h \left(b q_0^2 q_1 + c ( 2 p_0 q_0 q_1 + p_1 q_0^2) + d ( 2 p_0 p_1 q_0 + p_0^2 q_1) + e p_0^2 p_1\right)}{1 - 4 h^2 \Delta} \cr
p_2 &= \frac{p_0 - 2 h \left(a q_0^2 q_1 + b(2 p_0 q_0 q_1 + p_1 q_0^2) + c(2 p_0 p_1 q_0 + p_0^2 q_1) + d p_0^2 p_1\right)}{1 - 4 h^2 \Delta},
}
\end{equation}
where
\def\mat#1#2#3#4{{\left| \matrix{ #1 & #2 \cr #3 & #4} \right| }}
\begin{equation*}
\eqalign{
\Delta  = 
\mat{c}{d}{d}{e}p_0^2 p_1^2
& + \mat{b}{c}{d}{e} (p_0^2 p_1 q_1 + p_0 p_1^2 q_0)
+ \mat{b}{c}{c}{d}(p_0^2 q_1^2+p_1^2 q_0^2) \cr
& +\mat{a}{b}{c}{d}(p_1 q_0^2 q_1 + p_0 q_0 q_1^2)
+\mat{a}{c}{c}{e}p_0 p_1 q_0 q_1
+\mat{a}{b}{b}{c}q_0^2 q_1^2.}
\end{equation*}
The map is birational of degree 3 over degree 4.
The two 2-integrals are 
\begin{equation*}
q_0 p_1 - q_1 p_0\hbox{\rm \ and\ } q_1 p_2 - p_1 q_2,
\end{equation*}
where $q_2$ and $p_2$ are given in \eqref{eq:explicit}.
The invariant measure is
\begin{equation*}
\frac{dq_0 \wedge dp_0 \wedge dq_1 \wedge dp_1}{1 - 4 h^2 \Delta}.
\end{equation*}
 \end{example}

If the degree $k>2$ or the dimension $n>2$ then the geometric properties described above are not enough to ensure integrability. Indeed, we find that the polar map associated with a homogeneous planar quintic Hamiltonian (i.e. $k=3$, $n=2$) does not pass the entropy test for complete integrability \cite{bellon1999algebraic,veselov1992growth}.

\section*{Acknowledgements}
This research was supported by a  \href{http://wiki.math.ntnu.no/crisp}{Marie Curie International Research Staff Exchange Scheme Fellowship} within the \href{http://cordis.europa.eu/fp7/home_en.html}{7th European Community Framework Programme}, and by the Marsden Fund of the Royal Society of New Zealand and the Australian Research Council.

\end{document}